\documentclass[12pt]{article}

\usepackage{amsmath}
\usepackage{amssymb}
\usepackage{enumerate}

\allowdisplaybreaks
\newtheorem{theorem}{Theorem}

\newtheorem{corollary}[theorem]{Corollary}

\newtheorem{definition}[theorem]{Definition}

\newtheorem{lemma}[theorem]{Lemma}

\newtheorem{proposition}[theorem]{Proposition}

\newenvironment{proof}[1][Proof]{\noindent\textbf{#1.} }{\ \rule{0.5em}{0.5em}}
\numberwithin{equation}{section} 
\numberwithin{theorem}{section}

\begin{document}

\title{\bf{Solution of the Dirichlet problem for the equation $a\Delta u+b\cdot
\nabla u=0$ by the Monte Carlo method}}

\author{\bf{Jos\'{e} Villa-Morales}\\
Departamento de Matem\'{a}ticas y F\'{\i}sica,\\
Universidad Aut\'onoma de Aguascalientes,\\
Av. Universidad 940, C.P. 20131,\\
Aguascalientes, Ags., Mexico.\\
\texttt{jvilla@correo.uaa.mx}}
\date{ }
\maketitle

\begin{abstract}
In this paper we study the Dirichlet problem corresponding to an open
bounded set $D\subset \mathbb{R}^{d}$ and the operator
\begin{equation*}
A=\sum_{i=1}^{d}a\frac{\partial ^{2}}{\partial x_{i}^{2}}
+\sum_{i=1}^{d}b_{i}\frac{\partial }{\partial x_{i}},
\end{equation*}
where $a>0$ and $b\in \mathbb{R}^{d}$. We define a mean value property
and prove that a function $u$ has such property in $D$ if and only if $Au=0$
in $D$. Using this characterization, and a drifted Brownian motion, we define a family of random variables that converges almost surely and the limit is used to give an explicit representation for the solutions to the Dirichlet problem, this immediately implies the uniqueness. On the other hand, the existence of the solution  is proved imposing a regular condition on the boundary of $D$.
\end{abstract}

\hspace{-0.8cm} \textbf{Keywords.} Dirichlet problem, mean value property, harmonic functions, von Mises-Fisher distribution.

\section{Introduction}

As usual, by $(\mathbb{R}^{d},||\cdot ||)$ we are going to denote the
Euclidean norm space. For $G\subset \mathbb{R}^{d}$ we denote by $\overline{G}$ and $\partial G$ the closure and boundary (or frontier) of $G$, respectively.

Given a non-empty, bounded, and open subset $D$ of $\mathbb{R}^{d}$ and a
continuous function $f:\partial D\rightarrow \mathbb{R}$, we are interested
in find a unique continuous function $u:\overline{D}\rightarrow \mathbb{R}$
such that 
\begin{equation*}
u(x)=f(x),\ \ \forall x\in \partial D,
\end{equation*}
and moreover the function $u$ should have second partial derivatives on $D$
which satisfies the partial differential equation
\begin{equation}
\sum_{i=1}^{d}a\frac{\partial ^{2}f}{\partial x_{i}^{2}}(x)+%
\sum_{i=1}^{d}b_{i}\frac{\partial f}{\partial x_{i}}(x)=0,\ \ \forall x\in D,
\label{eqdp}
\end{equation}%
where $a>0$ and $(b_{1},...,b_{d})\in \mathbb{R}^{d}\backslash \{0\}$.

Such question is knwon in the literature as the Dirichlet problem. It has a
long history in pure and applied mathematics (see \cite{Ke}, \cite{Ka}, \cite{G-T}),
and there are a variety of ways to solve such problem. For example,
it can be solved by functional analysis techniques (see \cite{Fried}) or using complex
analysis (see \cite{Ahlfors}).

Here we are interested in solve the Dirichlet problem by probabilistic
techniques. The interplay between partial differential equations and
probability theory is an old subject and it was initiated by Kakutani \cite{Ka}. On
the other hand, Metropolis and Ulam introduced a statistical sampling
technique, called the Monte Carlo method, for solving physical problems.
Such method is very helpful and there are many studies of the Dirichlet
problem using Monte Carlo techniques (see, for example,  \cite{N-O}
and the references there in).

In the solution of the Dirichlet problem that we are going to present here is necessary a family of random variables. Next we describe how such family is
constructed. But before to do it we introduce some notation. By $d(x,G)$ we
design the distance from the point $x\in \mathbb{R}^{d}$ to the set $GA\subset \mathbb{R}^{d}$, to be precise
\begin{equation*}
d(x,G)=\inf \{||x-y||:y\in G\}.
\end{equation*}%
Let $B_{r}(x)=\{z\in \mathbb{R}^{d}:||x-z||<r\}$ be the open ball of radius 
$r>0$ centered at $x\in \mathbb{R}^{d}$, and $\partial B_{r}(x)=\{z\in 
\mathbb{R}^{d}:||x-z||=r\}$.

Let $W^{x}$ be a Brownian motion, defined on a probability space $(\Omega,\mathcal{F},\mathbb{P})$, that starts at $x\in \mathbb{R}^{d}$ and
consider the stochastic process $X^{x}$ defined as%
\begin{equation*}
X_{t}^{x}=tb+aW_{t}^{x/a},\ \ t\geq 0,
\end{equation*}%
where $a>0$, $b=(b_{1},...,b_{d})$. The stochastic process $X^{x}$ is continuous
and has the strong Markov property inherited from $W^{x}$. The strong Markov property of $X$ will be used frequently hereinafter, and it intuitively means that we can begins afresh the stochastic
process $X$ at stopping times.

Given $x\in D$ we would like to find the corresponding value $u(x)$. To this
end we take $0<\varsigma \leq 1$, arbitrary and fix, and construct the
sequence $(Y_{\varsigma }^{x}(n))_{n}$ as follows: $Y_{\varsigma }^{x}(1)=x$%
, let us run the process $X$ starting at $Y_{\varsigma }^{x}(1)$ and we stopped the
first time it exits the ball $B_{\varsigma d(Y_{\varsigma }^{x}(1),\partial
D)}(Y_{\varsigma }^{x}(1))$, then $Y_{\varsigma
}^{x}(2)$ is defined as the place where the stochastic process $%
X^{Y_{\varsigma }^{x}(1)}$ exits the ball $B_{\varsigma d(Y_{\varsigma
}^{x}(1),\partial D)}(Y_{\varsigma }^{x}(1))$, we restart again the process $%
X$ starting now at the point $Y_{\varsigma }^{x}(2)$, then we define $%
Y_{\varsigma }^{x}(3)$ as the place where the stochastic process $%
X^{Y_{\varsigma }^{x}(2)}$ exits the ball $B_{\varsigma d(Y_{\varsigma
}^{x}(2),\partial D)}(Y_{\varsigma }^{x}(2))$. Proceeding in this way we
obtain the desired sequence. We will prove in Lemma \ref{Lcs} that the
sequence $(Y_{\varsigma }^{x}(n))_{n}$ converges a.s. to a point $%
Y_{\varsigma }^{x}(\infty )\in \partial D$. This allow us to define the function $u(x)=f(Y_{\varsigma }^{x}(\infty))$, $x\in D$.

Muller in \cite{Muller} introduced this sequence taken $\varsigma =1$, when $a=1$
and $b=0$. Muller also proved the convergence of the sequence $(Y_{1}^{x}(n))_{n}$
based on the fact that such sequence has the same distribution as that of a
certain sequence which depends on the Brownian paths starting at $x$. It was
observed in \cite{Villa} that the ambiguity of the parameter $\varsigma $ is
the key to shows that $u(x)=f(Y_{\varsigma }^{x}(\infty ))$, $x\in D$, has the mean
value property (see (\ref{der}) ahead). In the case $A=\Delta$ it is well known that the mean value property is a useful characterization of harmonicity. Now we are going to introduce the new version of this concepts, since them will play a fundamental role.

By $I_{v}(z)$ we denote the modified Bessel function of the first kind
defined as
\begin{equation*}
I_{v}(z)=\frac{\left( \frac{z}{2}\right) ^{v}}{\pi ^{\frac{1}{2}}\Gamma
\left( v+\frac{1}{2}\right) }\int_{-1}^{1}(1-t^{2})^{v-\frac{1}{2}}e^{\pm
zt}dt,\ \ \forall z,v\in \mathbb{C}\text{ and Re }v>-\frac{1}{2}.
\end{equation*}
$I_{v}(z)$ is real and positive when $v>-1$ and $z>0$, see Section 9.6.1 and
formula 9.6.18 in \cite{A-S}.

\begin{definition}\label{defprinpvm}
Let $D$ be an open set. A function $u:D\rightarrow \mathbb{R}$ has the mean
value property in $D$ if $u$ is locally integrable and for all $x\in D$ and
all $r<d(x,\partial D)$,
\begin{equation}
u(x)=\kappa \left( \frac{r||b||}{2a}\right) \int_{\partial B_{r}(x)}u(y)\exp
\left\{ \frac{1}{2a}b\cdot (y-x)\right\} \mu _{r}(dy), \label{defpvm}
\end{equation}
where 
\begin{equation}
\kappa (z)=\left( \frac{z}{2}\right) ^{\frac{d}{2}-1}\frac{1}{\Gamma \left( 
\frac{d}{2}\right) I_{\frac{d}{2}-1}\left( z\right) },\ \ z>0,  \label{defk}
\end{equation}%
$\mu _{r}(dy)$ is the Lebesgue surface on $\partial B_{r}(x)$, normalized to
have total mass $1$.
\end{definition}

If $u$ has the mean value property, then the Dominated Convergence Theorem
implies%
\begin{eqnarray*}
\lim_{h\rightarrow 0}u(x+h) &=&\kappa \left( \frac{r||b||}{2a}\right)
\lim_{h\rightarrow 0}\exp \left\{ -\frac{1}{2a}b\cdot (x+h)\right\} \\
&&\times \lim_{h\rightarrow 0}\int_{\partial B_{r}(x+h)}u(y)\exp \left\{ 
\frac{1}{2a}b\cdot y\right\} \mu _{r}(dy) \\
&=&u(x).
\end{eqnarray*}
Therefore a function with the mean value property is continuous. As in the case $a=1$ and $b=0$ we have more, as
we will see in Theorem \ref{TeopvmH}, a function $u$ has the mean value
property if and only if it is harmonic, in the following sense.

\begin{definition}\label{defharmo}
Let $D$ be an open set. A function $u:D\rightarrow \mathbb{R}$ is called
harmonic in $D$ if $u$ is of class $C^{2}$ and $Af=0$ in $D$, where%
\begin{eqnarray}
Af(x) &=&a\Delta f(x)+b\cdot \nabla f(x) \nonumber\\ 
&=&\frac{\sigma ^{2}}{2}\sum_{i=1}^{d}\frac{\partial ^{2}f}{\partial
x_{i}^{2}}(x)+\sum_{i=1}^{d}b_{i}\frac{\partial f}{\partial x_{i}}(x),\label{openeds}
\end{eqnarray}
here $\sigma =\sqrt{2a}$ and $b=(b_{1},...,b_{n}).$
\end{definition}

Notice that from (\ref{openeds}) we recognize that $A$ is the infinitesimal generator of the strong Markov process $X$.

The expression $u(x)=f(Y_{\varsigma }^{x}(\infty ))$, $x\in D$, gives the uniqueness of
the Dirichlet problem, in fact we will see in Theorem \ref{ThUniqueness} that any other solution
has this representation. But the study of the existence of the solution is
not so easy. Actually, in the case $a=1$ and $b=0$, Zaremba observed in \cite{Za} that the
Dirichlet problem is not always solvable. In this case we need to impose
some regularity condition on the boundary of $D$.

\begin{definition}\label{regpoint}
Let $D$ be an open set and $\varsigma \in (0,1]$. We say that $v\in \partial
D$ is a regular point for $(D,f)$ if 
\begin{equation*}
\lim_{\substack{ x\rightarrow v  \\ x\in D}}\mathbb{E}\left[ f(Y_{\varsigma
}^{x}(\infty ))\right] =f(v),
\end{equation*}
where $\mathbb{E}$ is the expectation with respect to $\mathbb{P}$.
\end{definition}

In our main result (Theorem \ref{MainTh}) we will see that this regularity
condition ensure the existence of the Dirichlet problem.

It is worth mention that a second steep in this work could be a
computational implementation using the probability distribution of the
random sequence $(Y_{\varsigma }^{x}(n))_{n}$, see the identity (\ref{distrYn}). When $\varsigma=(d(Y_{\varsigma}^{0}(1),\partial D))^{-1}$ the random variable $Y_{\varsigma}^{0}(1)$ has the von Mises-Fisher distribution, see \cite{Gatto}.

The paper is organized as follows. In Section \ref{SecPreli} we prove that a
function has the mean value property if and only if it is harmonic. Using
the Convergence Theorem for discrete martingales we prove, in Section \ref%
{SecMteCarlo}, that the sequence $(Y_{\varsigma }^{x}(n))_{n}$ converges
a.s. Also in this section we prove the uniqueness and existence of the
Dirichlet problem, and as an easy application of the uniqueness we prove a
maximum principle, we also give in this section a criterion to determine when a point is
regular.

\section{Preliminaries\label{SecPreli}}

We begin by recalling that the Lebesgue integral of a function $f$ over $%
B_{r}(0)$ can be written in iterated form as%
\begin{equation}
\int_{B_{r}(0)}f(x)dx=\int_{0}^{r}\int_{\partial B_{s}(0)}f(x)\mu _{s}(dx)ds.
\label{intradl}
\end{equation}

\begin{lemma}
\label{pvmdif}If $u$ has the mean value property in $D$ (see Definition \ref{defprinpvm}), then $u$ is $%
C^{\infty }$ in $D$.
\end{lemma}

\begin{proof}
Since $I_{v}(z)$ is a holomorphic function of $z$ through the $z$-plane,
cut along the negative real axis, we deduce that $\kappa $ is a continuous
function. Moreover, we have (see \cite{A-S}, formula 9.6.7)
\begin{equation*}
I_{v}(z)\sim \left( \frac{z}{2}\right) ^{v}\frac{1}{\Gamma (v+1)},\ \ \text{as }z\rightarrow 0,
\end{equation*}
then $\lim_{z\downarrow 0}\kappa (z)=1$. This implies that for each $\varepsilon >0$ the integral 
\begin{equation*}
\int_{0}^{\varepsilon }\kappa \left( \frac{s||b||}{\sigma ^{2}}\right)
^{-1}\exp \left\{ \frac{1}{s^{2}-\varepsilon ^{2}}\right\} ds
\end{equation*}
is well defined, let us denote its value by $(c(\varepsilon ))^{-1}$. Let us also define the function $h_{\varepsilon }:\mathbb{R}^{d}\rightarrow \lbrack 0,\infty )$, as
\begin{equation*}
h_{\varepsilon }(x)=c(\varepsilon )g_{\varepsilon }(x)\exp \left\{ \frac{1}{\sigma ^{2}}b\cdot x\right\} ,
\end{equation*}
where
\begin{equation*}
g_{\varepsilon }(x)=\left\{ 
\begin{array}{ll}
\exp \left\{ \frac{1}{||x||^{2}-\varepsilon ^{2}}\right\} , & 
||x||<\varepsilon , \\ 
0, & ||x||\geq \varepsilon .
\end{array}
\right.
\end{equation*}
Since $g_{\varepsilon }$ is a $C^{\infty }$ function, then $h_{\varepsilon }$
is $C^{\infty }$ with support $\overline{B_{\varepsilon }(0)}$.

For each $\varepsilon >0$ and $x\in D$ such that $\overline{B_{\varepsilon}(x)}\subset D$, define the convolution
\begin{equation*}
(u\ast h_{\varepsilon })(x)=\int_{\mathbb{R}^{d}}u(y)h_{\varepsilon
}(x-y)dy=\int_{B_{\varepsilon }(0)}u(y+x)h_{\varepsilon }(y)dy.
\end{equation*}
Using (\ref{intradl}) the above expression can be written as, for $0<\varepsilon<1$,
\begin{align*}
&(u\ast h_{\varepsilon })(x)\\ 
&=\int_{0}^{\varepsilon }\int_{\partial B_{s}(0)}u(y+x)h_{\varepsilon}(y)\mu _{s}(dx)ds \\
&=\int_{0}^{\varepsilon }\int_{\partial B_{s}(0)}u(y+x)c(\varepsilon )
\exp \left\{ \frac{1}{\sigma ^{2}}b\cdot y\right\} \exp \left\{ \frac{1}{||y||^{2}-\varepsilon ^{2}}\right\} \mu _{s}(dy)ds \\
&=\int_{0}^{\varepsilon }\exp \left\{ \frac{1}{s^{2}-\varepsilon ^{2}}
\right\} c(\varepsilon )\int_{\partial B_{s}(0)}u(y+x)\exp \left\{ \frac{1}{\sigma ^{2}}b\cdot y\right\} \mu _{s}(dy)ds \\
&=\int_{0}^{\varepsilon }\exp \left\{ \frac{1}{s^{2}-\varepsilon ^{2}}
\right\} c(\varepsilon )\int_{\partial B_{s}(x)}u(y)\exp \left\{ \frac{1}{\sigma ^{2}}b\cdot (y-x)\right\} \mu _{s}(dy)ds,
\end{align*}
and, the mean value property of $u$, (\ref{defpvm}) yields
\begin{eqnarray*}
(u\ast h_{\varepsilon })(x)&=&u(x)c(\varepsilon )\int_{0}^{\varepsilon }\kappa \left( \frac{s||b||}{\sigma ^{2}}\right) ^{-1}\exp \left\{ \frac{1}{s^{2}-\varepsilon ^{2}}
\right\} ds.
\end{eqnarray*}

By the definition of $c(\varepsilon )$ we have $u=u\ast h_{\varepsilon }$.
Inasmuch as $h_{\varepsilon }$ is $C^{\infty }$ then $u=u\ast h_{\varepsilon
}$ is a $C^{\infty }$ function on the open subset of $D$ where it is
defined.\hfill
\end{proof}

\bigskip

Let $\mathbb{P}$ be the Wiener measure on $(\Omega ,\mathcal{F})$, where $\Omega =C[0,\infty )^{d}$ and $\mathcal{F}=\mathcal{B(}C[0,\infty )^{d}\mathcal{)}$. The coordinate proceses $W_{t}(\omega )=\omega (t)$, $\omega\in \Omega $, is the $d$-dimensional Brownian motion on $(\Omega ,\mathcal{F}
,\mathcal{\mathbb{P}})$, starting at $0$. If $x\in \mathbb{R}^{d}$, then $W^{x}=x+W$ will be the $d$-dimensional Brownian motion starting at $x$.

Let us consider the stopping time%
\begin{equation*}
\tau _{B_{r}(x)}^{W}=\inf \{t>0:W^{x}_{t}\notin B_{r}(x)\},
\end{equation*}
which is the first time the process $W^{x}$ exits the ball $B_{r}(x)$.

The $d$-dimensional Brownian motion has many important properties. It is a
martingale and a Markov process, moreover it has the strong Markov property, as we remarked this means that the process begins afresh at stopping times. The translation and
rotational invariance of Brownian motion implies that (see Proposition I.2.8
in \cite{Bass})%
\begin{equation}
\mathbb{P}(x+W_{\tau (B_{r}(x))}\in dy)=\mu _{r}(dy),  \label{disuniB}
\end{equation}
recall that $\mu _{r}(dy)$ is the Lebesgue surface on $\partial B_{r}(x)$,
normalized to have total mass $1$.

In order to find the Laplace transform of the distribution of $\tau^{W}(B_{r}(0))$ let us consider the distance from the Brownian motion $W$ to the origin $0$,
\begin{equation*}
R_{t}=||W_{t}||,\ \ 0\leq t<\infty .
\end{equation*}
This define the process $R$ called Bessel process. If we denote by $\tau _{r}$ the
first hitting time to $r>0$ of the Bessel process, by general theory of
one-dimensional diffusion processes we can evaluate the Laplace transform of
the distribution of $\tau _{r}$\ by solving an eigenvalue problem. In fact,
denoting by $\mathbb{E}$ the expectation with respect to $\mathbb{P}$, we have (see \cite{Kent} or \cite{Gatto})
\begin{equation}
\mathbb{E}\left[ e^{-\lambda \tau _{r}}\right] =\kappa \left( r\sqrt{%
2\lambda }\right) ,  \label{transLB}
\end{equation}%
where $\kappa $ is defined in (\ref{defk}).

In what follows we are going to consider the process $X$, 
\begin{equation}
X_{t}^{x}=x+bt+\sigma W_{t},\ \ t\geq 0.  \label{procX}
\end{equation}
Such process will be the basic stochastic object to deal with the Dirichlet problem for the operator $A$.

\begin{proposition}
\label{pvmExp}A local integrable function $u:D\rightarrow\mathbb{R}$ has the mean value property in $D$ if and if
\begin{equation*}
\mathbb{E}\left[ u\left( X_{\tau ^{X}(B_{r}(x))}^{x}\right) \right] =u(x),
\end{equation*}%
for all $x \in D$, and for all $r<d(x,\partial D)$.
\end{proposition}

\begin{proof}
We set
\begin{equation*}
Z_{t}=\exp \left\{ -\frac{1}{\sigma }b\cdot W_{t}-\frac{1}{2\sigma ^{2}}%
||b||^{2}t\right\} ,\ \ t\geq 0.
\end{equation*}
By Novikov condition (Corollary 5.13 in \cite{K-S}) the stochastic process $Z$ is a martingale. Then
Girsanov theorem (Theorem 5.1 in \cite{K-S}) implies the process $ \tilde{W}$, defined as,
\begin{equation*}
\tilde{W}_{t}=W_{t}+\frac{t}{\sigma }b,\ \ t\geq 0,
\end{equation*}
is a $d$-dimensional Brownian motion on $(\Omega ,\mathcal{F},\mathcal{%
\mathbb{\tilde{P}}})$, where the probability measure $\mathcal{\mathbb{%
\tilde{P}}}$ satisfies%
\begin{equation*}
\mathcal{\mathbb{\tilde{P}}}(A)\mathcal{\mathbb{=E}}[1_{A}Z_{T}],\ \ \forall
A\in \mathcal{F}_{T},\ \ 0\leq T<\infty .
\end{equation*}%
In particular 
\begin{equation}
\frac{d\mathbb{P}}{d\mathbb{\tilde{P}}}=\exp \left\{ \frac{1}{\sigma }b\cdot
W_{T}+\frac{1}{2\sigma ^{2}}||b||^{2}T\right\}, \text{ on }\mathcal{F}_{T}. \label{rnd}
\end{equation}
Observe that%
\begin{eqnarray}
\tau _{B_{r}(x)}^{X} &=&\inf \{t>0:||X_{t}^{x}-x||\geq r\}  \notag \\
&=&\inf \left\{ t>0:\left\Vert \frac{t}{\sigma }b+W_{t}\right\Vert \geq 
\frac{r}{\sigma }\right\}  \notag \\
&=&\inf \left\{ t>0:||\tilde{W}_{t}||\geq \frac{r}{\sigma }\right\} =\tau ^{%
\tilde{W}}(B_{r/\sigma }(0)).  \label{tpxmt}
\end{eqnarray}%
Let $h:\partial (B_{r}(x))\rightarrow \mathbb{R}$ be any bounded $\mathcal{B}%
(\partial (B_{r}(x)))$-$\mathcal{B}(\mathbb{R})$ measurable function, then (\ref{rnd}) brings about
\begin{align*}
&\mathcal{\mathbb{E}}\left[ h\left( X_{\tau ^{X}(B_{r}(x))}^{x}\right) \right]\\
&=\mathcal{\mathbb{\tilde{E}}}\left[ h\left( X_{\tau
^{X}(B_{r}(x))}^{x}\right) \left. \frac{d\mathbb{P}}{d\mathbb{\tilde{P}}}\right\vert _{\tau ^{X}(B_{r}(x))}\right] \\
&=\mathcal{\mathbb{\tilde{E}}}\left[ h\left( X_{\tau
^{X}(B_{r}(x))}^{x}\right) \exp \left\{ \frac{1}{\sigma }b\cdot W_{\tau ^{X}(B_{r}(x))}+\frac{||b||^{2}}{2\sigma ^{2}}\tau ^{X}(B_{r}(x))\right\} \right] \\
&=\mathcal{\mathbb{\tilde{E}}}\left[ h\left( x+\sigma \tilde{W}_{\tau ^{X}(B_{r}(x))}\right) \exp \left\{ \frac{1}{\sigma }b\cdot \tilde{W}_{\tau ^{X}(B_{r}(x))}-\frac{||b||^{2}}{2\sigma ^{2}}\tau ^{X}(B_{r}(x))\right\} \right] .
\end{align*}
From (\ref{tpxmt}), and using that $\tilde{W}_{\tau ^{\tilde{W}}(B_{r/\sigma
}(0))}$ and $\tau ^{\tilde{W}}(B_{r/\sigma }(0))$ are $\mathbb{\tilde{P}}$%
-independent we obtain
\begin{align*}
& \mathcal{\mathbb{E}}\left[ h\left( X_{\tau ^{X}(B_{r}(x))}^{x}\right) %
\right] \\
& =\mathcal{\mathbb{\tilde{E}}}\left[ h\left( x+\sigma \tilde{W}_{\tau ^{%
\tilde{W}}(B_{r/\sigma }(0))}\right) \exp \left\{ \frac{1}{\sigma }b\cdot 
\tilde{W}_{\tau ^{\tilde{W}}(B_{r/\sigma }(0))}-\frac{||b||^{2}}{2\sigma ^{2}%
}\tau ^{\tilde{W}}(B_{r/\sigma }(0))\right\} \right] \\
& =\mathcal{\mathbb{\tilde{E}}}\left[ h\left( x+\sigma \tilde{W}_{\tau ^{%
\tilde{W}}(B_{r/\sigma }(0))}\right) \exp \left\{ \frac{1}{\sigma }b\cdot 
\tilde{W}_{\tau ^{\tilde{W}}(B_{r/\sigma }(0))}\right\} \right] \\
& \times \mathcal{\mathbb{\tilde{E}}}\left[ \exp \left\{ -\frac{||b||^{2}}{%
2\sigma ^{2}}\tau ^{\tilde{W}}(B_{r/\sigma }(0))\right\} \right] .
\end{align*}%
The equality $\tau _{r/\sigma}=\tau ^{\tilde{W}}(B_{r/\sigma }(0))$ allows us to use (\ref{disuniB}). Then (\ref{transLB}) yields
\begin{align}
&\mathcal{\mathbb{E}}\left[ h\left( X_{\tau ^{X}(B_{r}(x))}^{x}\right) \right] \notag\\
&=\int_{\partial (B_{r/\sigma }(0))}h(x+\sigma y)\exp \left\{ \frac{1}{\sigma }b\cdot y\right\} \mathbb{\tilde{P}}\left( \tilde{W}_{\tau ^{\tilde{W}}(B_{r/\sigma }(0))}\in dy\right) \kappa \left( \frac{r||b||}{\sigma ^{2}}\right)  \notag \\
&=\int_{\partial (B_{r}(x))}h(y)\exp \left\{ \frac{1}{\sigma ^{2}}b\cdot
(y-x)\right\} \mu _{r}(dy)\kappa \left( \frac{r||b||}{\sigma ^{2}}\right) .
\label{extvalea}
\end{align}
This means $\mathbb{P}\left( X_{\tau ^{X}(B_{r}(x))}^{x}\in dy\right) $ is
absolutely continuous with respect to $\mu _{r}(dy)$ and its density given by%
\begin{equation*}
\frac{\mathbb{P}\left( X_{\tau ^{X}(B_{r}(x))}^{x}\in dy\right) }{\mu
_{r}(dy)}=\kappa \left( \frac{r||b||}{\sigma ^{2}}\right) \exp \left\{\frac{1}{\sigma ^{2}}b\cdot (y-x)\right\} .
\end{equation*}
This fact immediately implies the result.\hfill
\end{proof}

\begin{theorem}\label{TeopvmH}
A function $u:D\rightarrow \mathbb{R}$ has the mean value property in $D$ if and only
if it is harmonic in $D$ (see Definition \ref{defharmo}).
\end{theorem}

\begin{proof}
Let us suppose $u$ has the mean value property in $D$, hence Lemma \ref%
{pvmdif} implies $u$ is $C^{\infty }$. Suppose $Au(x_{0})>0$ for some $%
x_{0}\in D$. The continuity of $Au$ implies there exists $r<d(x_{0},\partial D)$ such that $Au>0$ on $B_{r}(x_{0})$. Let us consider the
stochastic process $X=X^{x_{0}}$, defined at (\ref{procX}). By It\^{o}'s
formula (see Theorem 3.6 in \cite{K-S}) we obtain
\begin{equation}
u\left( X_{t\wedge \tau ^{X}(B_{r}(x_{0}))}\right) -u(X_{0})=\text{%
martingale }+\frac{1}{2}\int_{0}^{t\wedge \tau
^{X}(B_{r}(x_{0}))}Au(X_{s})ds.  \label{aplIto}
\end{equation}%
Now taking expectations and letting $t\rightarrow \infty $ we get%
\begin{equation*}
\mathbb{E}\left[ u\left( X_{\tau ^{X}(B_{r}(0))}\right) \right] -u(x_{0})=%
\frac{1}{2}\mathbb{E}\left[ \int_{0}^{\tau ^{X}(B_{r}(x_{0}))}Au(X_{s})ds%
\right] >0.
\end{equation*}%
By Proposition \ref{pvmExp} we have that the lef hand side, of the above equality, is $0$. This
contradiction implies $Au(x_{0})\leq 0$. If we suppose $Au(x_{0})<0$ and
proceeding as before we deduce $Au(x_{0})\geq 0$. In this way, $Au=0$ in $D$.

Reciprocally, assume $u$ is harmonic in $D$. Let $x_{0}\in D$ and $r<d(x_{0},\partial D)$. By It\^{o}'s formula $u\left( X_{t\wedge \tau
^{X}(B_{r}(x_{0}))}\right) -u(X_{0})$ is a martingale, this is due to the
second term in (\ref{aplIto}) is $0$, since $Au=0$ in $D$. If we take
expectation we get $u(x_{0})=\mathbb{E}\left[ u\left( X_{t\wedge \tau
^{X}(B_{r}(x_{0}))}\right) \right] $, and letting $t\rightarrow \infty $ turns out $u(x_{0})=\mathbb{E}\left[ u\left( X_{\tau ^{X}(B_{r}(x_{0}))}\right) %
\right] $. Therefore, by Proposition \ref{pvmExp}, the function $u$ has the
mean value property.\hfill
\end{proof}

\section{The Monte Carlo method\label{SecMteCarlo}}

Let $\varsigma \in (0,1]$ be fix. For every $x\in D \subset \mathbb{R}^{d}$, we define the sequence 
$(Y_{\varsigma }^{x}(n))_{n}$ as 
\begin{eqnarray}
Y_{\varsigma }^{x}(1) &=&x,  \notag \\
Y_{\varsigma }^{x}(n+1) &=&X_{\tau ^{X}\left( B_{r_{n}}(Y_{\varsigma
}^{x}(n))\right) }^{Y_{\varsigma }^{x}(n)},\ \ n\geq 1,  \label{defseq}
\end{eqnarray}%
where $r_{n}=\varsigma d(Y_{\varsigma }^{x}(n),\partial D)$. The state space of the random
variable $Y_{\varsigma }^{x}(n+1)$ is $\partial
B_{r_{n}}\left( Y_{\varsigma }^{x}(n)\right) $ and the strong Markov
property of $X$ implies%
\begin{equation}
\mathbb{P}\left( Y_{\varsigma }^{x}(n+1)\in dz|Y_{\varsigma
}^{x}(n)=y\right) =\kappa \left( \frac{r_{n}||b||}{\sigma ^{2}}\right) \exp
\left\{ \frac{r_{n}}{\sigma ^{2}}b\cdot \left( z-y\right) \right\} \mu
_{r_{n}}(dz),  \label{distrYn}
\end{equation}%
we hold this for each $n\in \mathbb{N}$.

\begin{lemma}
\label{ArmoMartinga} If $g:\overline{D}\rightarrow \mathbb{R}$ is a continuous function 
and has the mean value property in $D$, then for each $x\in D$ the sequence $(g(Y_{\varsigma }^{x}(n)))_{n}$ is a martingale with respect to $\mathcal{F}_{n}=\sigma (Y_{\varsigma
}^{x}(1),...,Y_{\varsigma }^{x}(n))$, which is the minimal $\sigma $-algebra
such that $Y_{\varsigma }^{x}(1),...,Y_{\varsigma }^{x}(n)$ are measurable.
\end{lemma}

\begin{proof}
First observe that, for each $n\in \mathbb{N}$, $Y_{\varsigma }^{x}(n+1)\in \overline{D}$, this is because $Y_{\varsigma }^{x}(n+1)\in \partial B_{r_{n}}\left( Y_{\varsigma}^{x}(n)\right) $. Since $\overline{D}$ is compact and $g$ continuous in $\overline{D}$, then $(g(Y_{\varsigma
}^{x}(n)))_{n}$ is an integrable sequence of random variables. Using the
strong Markov property of $X$ (see Proposition 2.6.6 in \cite{K-S}) we
obtain 
\begin{eqnarray*}
\mathbb{E}\left[ g(Y_{\varsigma }^{x}(n+1))|\mathcal{F}_{n}\right] &=&%
\mathbb{E}\left[ g(Y_{\varsigma }^{x}(n+1))|Y_{\varsigma }^{x}(n)\right] \\
&=&\left. \mathbb{E}\left[ g\left( X_{\tau ^{X}\left( B_{r_{n}}(y)\right)
}^{y}\right) \right] \right\vert _{y=Y_{\varsigma }^{x}(n)} \\
&=&\left. g(y)\right\vert _{y=Y_{\varsigma }^{x}(n)}=g(Y_{\varsigma
}^{x}(n)).
\end{eqnarray*}%
To obtain the third equality we have used Proposition \ref{pvmExp}.\hfill
\end{proof}

\begin{lemma}
\label{Lcs} For each $x\in D$, the sequence $(Y_{\varsigma }^{x}(n))_{n}$ converges a.s. to a
point $Y_{\varsigma }^{x}(\infty )\in \partial D$.
\end{lemma}

\begin{proof}
Let $h_{j}:\overline{D}\rightarrow \mathbb{R}$ be defined as%
\begin{equation*}
h_{j}(x_{1},...,x_{d})=\left\{ 
\begin{array}{ll}
\exp \left\{ -\frac{2b_{j}}{\sigma ^{2}}x_{j}\right\} , & b_{j}\neq 0, \\ 
x_{j}, & b_{j}=0,%
\end{array}%
\right.
\end{equation*}%
for each $j\in \{1,...,d\}$. A direct calculation shows
\begin{equation}
Ah_{j}(x)=0,\ \ x\in D,  \label{Aaplh}
\end{equation}%
for each $j\in \{1,...,d\}$. Then Theorem \ref{TeopvmH} implies $h_{j}$\ has
the mean value property, hence $(h_{j}(Y_{\varsigma }^{x}(n)))_{n}$ is a
martingale, by Lemma \ref{ArmoMartinga}. Since $D$ is bounded then $(h_{j}(Y_{\varsigma }^{x}(n)))_{n}$ is a bounded martingale, therefore the Convergence Theorem for martingales implies $\lim_{n\rightarrow \infty}h_{j}(Y_{\varsigma }^{x}(n))=H_{\varsigma }^{x,j}(\infty )$ a.s. If $b_{j}\neq 0$ from the boundedness of $D$ we deduce $H_{\varsigma }^{x,j}(\infty )>0$
a.s. Then $(Y_{\varsigma }^{x}(n))_{n}$ converges a.s. to $Y_{\varsigma
}^{x}(\infty )=\left( Y_{\varsigma }^{x,1}(\infty ),...,Y_{\varsigma
}^{x,d}(\infty )\right) $, where
\begin{equation*}
Y_{\varsigma }^{x,j}(\infty )=\left\{ 
\begin{array}{ll}
-\cfrac{\sigma ^{2}}{2b_{j}}\log \left( H_{r}^{x,j}(\infty )\right) , & 
b_{j}\neq 0, \\ 
H_{\varsigma }^{x,j}(\infty ), & b_{j}=0,
\end{array}%
\right.
\end{equation*}%
for each $j\in \{1,...,d\}$.

On the other hand, since $Y_{\varsigma }^{x}(n+1)\in \partial B_{r_{n}}\left(
Y_{\varsigma }^{x}(n)\right) $ then
\begin{equation*}
||Y_{\varsigma }^{x}(n+1)-Y_{\varsigma }^{x}(n)||=\varsigma d(Y_{\varsigma
}^{x}(n),\partial D), \ \  \forall n\in \mathbb{N}.
\end{equation*}
Letting $n\rightarrow \infty $ we have%
\begin{equation*}
\varsigma d(Y_{\varsigma }^{x}(\infty ),\partial D)=||Y_{\varsigma
}^{x}(\infty )-Y_{\varsigma }^{x}(\infty )||=0,
\end{equation*}%
therefore, $Y_{\varsigma }^{x}(\infty )\in \partial D$, this is because $\partial D$
is a closed set.\hfill
\end{proof}

\begin{theorem}[Existence]\label{MainTh} 
Let $D\subset \mathbb{R}^{d}$ be an open bounded set and $f:\partial D\rightarrow \mathbb{R}$ be a continuous function. If every point in $\partial V$ is regular, then the function $u:\overline{D}\rightarrow 
\mathbb{R}$ defined as 
\begin{equation}
u(x)=\left\{ 
\begin{tabular}{ll}
$f(x),$ & $x\in \partial D,$ \\ 
$\mathbb{E}\left[ f(Y_{\varsigma }^{x}(\infty ))\right] ,$ & $x\in D,$%
\end{tabular}%
\ \right.  \label{defh}
\end{equation}%
is continuous on $\overline{D}$ and satisfies the partial differential equation 
\begin{equation}
\frac{\sigma ^{2}}{2}\Delta u(x)+b\cdot \nabla u(x)=0,\ \ \forall x\in D.
\label{EDP2}
\end{equation}
\end{theorem}

\begin{proof}
By $D_{r}$ we mean the set $\{x\in \mathbb{R}^{d}: d(x,D)<r\}$, $r>0$. The
Tietze-Urysohn theorem implies there exists a continuous function $%
\tilde{f}:D_{2}\rightarrow \mathbb{R}$ such that $\tilde{f}|_{\partial D}=f$%
. Hence, $\tilde{f}$ is bounded in $\overline{D_{1}}\subset D_{2}$, so $%
\tilde{f}\in L^{1}(D_{1})$. Let us take the function $h:\mathbb{R}%
^{d}\rightarrow \mathbb{R}$ defined as 
\begin{equation*}
h(x)=\left\{ 
\begin{array}{ll}
c\exp \left\{ -\frac{2}{\sigma ^{2}}b\cdot x\right\} , & x\in \overline{D_{1}%
}, \\ 
0, & x\notin \overline{D_{1}},%
\end{array}%
\right.
\end{equation*}%
where 
$$c^{-1}=\int_{\overline{D_{1}}} \exp\{-2\sigma ^{-2}(b\cdot x)\}dx.$$ If $\varepsilon >0$ we set%
\begin{equation*}
h_{\varepsilon }(x)=\frac{1}{\varepsilon ^{d}}h\left( \frac{1}{\varepsilon }%
x\right) ,\ \ x\in \mathbb{R}^{d}\text{.}
\end{equation*}%
Using that $A\left( \tilde{f}\ast h_{\varepsilon }\right) =\tilde{f}\ast A\left(
h_{\varepsilon }\right) $ and (\ref{Aaplh}) we conclude, by Theorem \ref{TeopvmH}, that the sequence $(\tilde{f}\ast h_{\varepsilon })_{\varepsilon >0}$ has the mean value property in $D_{1}$. Moreover, (see Theorem 8.14 in \cite{Folland}) 
\begin{equation*}
\lim_{\varepsilon \downarrow 0}(\tilde{f}\ast h_{\varepsilon })(x)=\tilde{f}%
(x)\text{, \ uniformly in }\overline{D}\subset D_{1}.
\end{equation*}%
Lemma \ref{ArmoMartinga} and the Dominated Convergence Theorem yield, for each $x\in D$, 
\begin{eqnarray*}
\mathbb{E}\left[ f(Y_{\varsigma }^{x}(\infty ))\right] &=&\mathbb{E}\left[ 
\tilde{f}(Y_{\varsigma }^{x}(\infty ))\right] \\
&=&\lim_{\varepsilon \downarrow 0}\mathbb{E}\left[ (\tilde{f}\ast
h_{\varepsilon })(Y_{\varsigma }^{x}(\infty ))\right] \\
&=&\lim_{\varepsilon \downarrow 0}\lim_{n\rightarrow \infty }\mathbb{E}\left[
(\tilde{f}\ast h_{\varepsilon })(Y_{\varsigma }^{x}(n))\right] \\
&=&\lim_{\varepsilon \downarrow 0}\lim_{n\rightarrow \infty }(\tilde{f}\ast
h_{\varepsilon })(x)=\tilde{f}(x).
\end{eqnarray*}%
This means the definition (\ref{defh}) of $u$ does not depend on $\varsigma$, so $u$ is well defined.

Let $x\in D$ and $r<d(x,\partial D)$. Let us take 
\begin{equation}
\varsigma =\frac{r}{d(x,\partial D)}\leq 1.  \label{der}
\end{equation}%
The definition (\ref{defseq}) implies $Y_{\varsigma }^{x}(n+1)=Y_{\varsigma
}^{Y_{\varsigma }^{x}(2)}(n)$, for all $n\geq 2$. The strong Markov property of $X$ (see Section 5.4 in \cite{Ash}) implies
\begin{equation*}
\mathbb{E}\left[ \tilde{f}(Y_{\varsigma }^{x}(n+1))|Y_{\varsigma }^{x}(2)%
\right] =\left. \mathbb{E}\left[ \tilde{f}(Y_{\varsigma }^{y}(n))\right]
\right\vert _{y=Y_{\varsigma }^{x}(2)},\text{ \ }\forall n\geq 2.
\end{equation*}%
Letting $n\rightarrow\infty$ in the above equality we have, by the Dominated Convergence Theorem for conditional expectations, 
\begin{equation*}
\mathbb{E}\left[ \tilde{f}(Y_{\varsigma }^{x}(\infty ))|Y_{\varsigma }^{x}(2)%
\right] =\left. \mathbb{E}\left[ \tilde{f}(Y_{\varsigma }^{y}(\infty ))%
\right] \right\vert _{y=Y_{\varsigma }^{x}(2)}.
\end{equation*}%
From (\ref{der}) we see that $Y_{\varsigma }^{x}(2)$ has values in $\partial
B_{r}(x)$, then (\ref{distrYn}) turns out
\begin{eqnarray*}
\mathbb{E}[f(Y_{\varsigma }^{x}(\infty ))] &=&\mathbb{E}\left[ \mathbb{E}%
[f(Y_{\varsigma }^{x}(\infty ))|Y_{\varsigma }^{x}(2)]\right] \\
&=&\mathbb{E}\left[ \left. \mathbb{E}\left[ f(X_{\varsigma }^{z}(\infty ))%
\right] \right\vert _{z=Y_{\varsigma }^{x}(2)}\right] \\
&=&\kappa \left( \frac{r||b||}{\sigma ^{2}}\right) \int_{\partial
B_{r}\left( x\right) }\mathbb{E}[f(X_{\varsigma }^{z}(\infty ))]\exp \left\{ 
\frac{r}{\sigma ^{2}}b\cdot (z-x)\right\} \mu _{r}(dz).
\end{eqnarray*}

This means the function $x\mapsto \mathbb{E}[f(Y_{\varsigma }^{x}(\infty ))]$
has the mean value property in $D$, then by Theorem \ref{TeopvmH} we
have that $u$ is harmonic in $D$. The continuity of $u$ follows immediately from the Definition \ref{regpoint} of regular points.\hfill
\end{proof}

A sufficient condition to analyze the regularity of the boundary points of $D $ is given through the following condition.

\begin{definition}
Let $v\in \partial D$. A continuous function $q_{v}:\overline{D}\rightarrow 
\mathbb{R}$ is called a barrier at $v$ if $q_{v}$ is harmonic in $D$, $%
q_{v}(v)=0$ and 
\begin{equation}
q_{v}(x)>0,\ \ \forall x\in \overline{D}\backslash \{v\}.  \label{limsp}
\end{equation}
\end{definition}

\begin{proposition}
Let $v\in \partial D$ be a point with a barrier $q_{v}$, then it is regular.
\end{proposition}

\begin{proof}
Let $M=\sup \{|f(x)|:x\in \partial D\}$. The continuity of $f$ in $\partial D$ implies that for each $\varepsilon >0$ there exists $\delta >0$ such that 
\begin{equation*}
x\in \partial D\text{, }||x-v||<\delta \Rightarrow |f(x)-f(v)|<\varepsilon .
\end{equation*}%
On the other hand, from (\ref{limsp}) we have 
\begin{equation*}
K=\inf \{q_{v}(z):||z-v||\geq \delta ,z\in \overline{D}\}>0,
\end{equation*}%
this allows us to get
\begin{equation*}
K^{-1}q_{v}(z)\geq 1,\ \ \forall z\in \overline{D},\ ||z-v||\geq \delta .
\end{equation*}%
Therefore%
\begin{equation*}
|f(x)-f(v)|\leq \varepsilon +(2MK^{-1})q_{v}(x)\text{, \ }\forall x\in
\partial D.
\end{equation*}%
Let $(v_{k})$ be an arbitrary sequence in $D$ such that $\lim_{k\rightarrow
\infty }v_{k}=v$. Define $Y_{\varsigma}^{v_{k}}(\infty )$ as we did in (\ref{defseq}%
). Lemma \ref{ArmoMartinga} implies,%
\begin{eqnarray*}
|f(v)-\mathbb{E}\left[ f(Y_{\varsigma }^{v_{n}}(\infty ))\right] | &=&|\mathbb{E}\left[ f(v) - f(Y_{\varsigma}^{v_{n}}(\infty ))\right] | \\
&\leq &\mathbb{E}\left[ |f(v)-f(Y_{\varsigma }^{v_{n}}(\infty ))|\right] \\
&\leq &\varepsilon +(2MK^{-1})\mathbb{E}\left[ q_{v}(Y_{\varsigma
}^{v_{n}}(\infty ))\right] \\
&=&\varepsilon +(2MK^{-1})q_{v}(v_{n}).
\end{eqnarray*}%
From the continuity of $q_{v}$ we get the desired result.\hfill
\end{proof}

We say that a point $v\in \partial D$ satisfies the Poincar\'{e}'s condition if
there exists a ball $B_{s}(u)\subset\mathbb{R}^{d}\backslash D$ such that $\overline{D}\cap \overline{B_{s}(u)}=\{v\}$. Next we will get an application of the barrier condition.

\begin{corollary}
If $a=1$ and $b=0$, then each point in $\partial D$, that satisfies the Poincar\'{e}'s condition, is a regular point.
\end{corollary}

\begin{proof}
Let $v\in \partial D$ for which there exists $B_{s}(u)\subset\mathbb{R}^{d}\backslash D$ such that $\overline{D%
}\cap \overline{B_{s}(u)}=\{v\}$, then $q_{v}:\overline{D}\rightarrow 
\mathbb{R}$, defined as, 
\begin{equation*}
q_{v}(x)=\left\{ 
\begin{array}{ll}
\log \left( \frac{||x-u||}{s}\right) , & d=2, \\ 
s^{2-d}-||x-u||^{2-d}, & d\geq 3,%
\end{array}%
\right.
\end{equation*}%
is a barrier at $v$.\hfill
\end{proof}

Now let us deal with the uniqueness of the Dirichlet problem, here we do not
need to assume any type of regularity on the boundary of $D$.

\begin{theorem}[Uniqueness]\label{ThUniqueness}
\label{MainTh copy(1)} Let $D\subset \mathbb{R}^{d}$ be an open bounded set
and $f:\partial D\rightarrow \mathbb{R}$ be a continuous function. There
exists at most one continuous function $u:\overline{D}\rightarrow \mathbb{R}$ such that $u|_{\partial D}=f$ and $u$ satisfies the partial differential equation (\ref{EDP2}) in $D$.
\end{theorem}

\begin{proof}
If $h:\overline{D}\rightarrow \mathbb{R}$ is a solution for the Dirichlet
problem, then $h$ is continuous in $\overline{D}$. For each $x\in D$ we have by Lemma \ref{Lcs}, 
\begin{equation*}
\lim_{n\rightarrow \infty }h(Y_{1}^{x}(n))=f(Y_{1}^{x}(\infty )),\text{ \
a.s.}
\end{equation*}%
Otherwise, Lemma \ref{ArmoMartinga} implies that $(h(Y_{1}^{x}(n)))_{n}$ is a
martingale and, by the Dominated Convergence Theorem, 
\begin{equation*}
h(x)=\mathbb{E}\left[ h(Y_{1}^{x}(1))\right] =\lim_{n\rightarrow \infty }%
\mathbb{E}\left[ h(Y_{1}^{x}(n))\right] =\mathbb{E}\left[ f(Y_{1}^{x}(\infty
))\right] .
\end{equation*}%
Therefore, $h(x)=E[f(Y_{1}^{x}(\infty ))]=u(x)$, for each $x\in D$. This identity
gives us the sought uniqueness of the Dirichlet problem.\hfill
\end{proof}

\begin{corollary}
If $D$ is a bounded open set and $u$ is harmonic in $D$ and continuous in $\overline{D}$, then 
\begin{equation}
\sup_{x\in \overline{D}}u(x)=\sup_{x\in \partial D}u(x).  \label{igsup}
\end{equation}
\end{corollary}

\begin{proof}
Let $x\in D$, then the previous result implies
\begin{eqnarray*}
u(x) &=&\mathbb{E}\left[ f(Y_{1}^{x}(\infty ))\right] \\
&\leq &\mathbb{E}\left[ \sup_{y\in \partial D}f(y)\right] =\sup_{y\in
\partial D}f(y).
\end{eqnarray*}%
From this (\ref{igsup}) follow readily.\hfill
\end{proof}

\bigskip
\noindent {\it Acknowledgments}\\
The research was partially supported through the grant PIM 14-4 of the Universidad Aut\'onoma de Aguascalientes.

\end{document}